\documentclass[12pt,a4paper]{amsart}
\usepackage{amscd}
\usepackage{graphicx}
\usepackage{datetime}
\newcommand{\tast}{T^{\ast}T}
\newcommand{\xast}{x^{\ast}}

\newcommand{\bastb}{B^{\ast}B}

\newcommand{\bart}{T}
\newcommand{\psmooth}{\beta}
\newcommand{\expect}{\mathbb E}
\newcommand{\ra}{r_{\alpha}}
\newcommand{\sa}{s_{\alpha}}
\newcommand{\ga}{g_{\alpha}}
\newcommand{\xad}{x^\delta_\alpha}
\newcommand{\Cad}{C^\delta(\alpha)}
\newcommand{\Vad}{V^\delta(\alpha)}

\newcommand{\lr}[1]{\left( #1 \right)}
\newcommand{\norm}[2]{\left\| #1 \right\|_{#2}}
\newcommand{\abs}[1]{| #1 |}
\newcommand{\set}[1]{\left\{ #1 \right\}}
\newcommand{\scalar}[2]{\langle{ #1},{#2} \rangle}
\newcommand{\real}{\mathbb R}
\newcommand{\yd}{y^\delta}
\newcommand{\af}{A_\varphi}
\newcommand{\aast}{\alpha_\ast}
\newcommand{\tr}[1]{\operatorname{tr}\left[#1\right]}
\newcommand{\range}{\mathcal R}

\newcommand{\domain}{\mathcal D}
\newcommand{\prc}{C_0}
\newcommand{\poc}{C}
\newcommand{\prm}{m_0}
\newcommand{\pom}{m}
\newcommand{\Sp}{S_{T,\prc}}
\newcommand{\spc}{\operatorname{SPC}}
\newcommand{\mad}{m_\alpha^\delta}

\newcommand{\G}{\mathcal N}
\newcommand{\cN}{\mathcal N}
\newcommand{\z}{z^\delta}
\newenvironment{frcseries}{\fontfamily{frc}\selectfont}{}

\newcommand{\bigo}{\mathcal O}
\newcommand{\prr}{a}
\newtheorem{lem}{Lemma}

\newtheorem{thm}{Theorem}
\newtheorem{prop}{Proposition}
\newtheorem{Def}{Definition}
\theoremstyle{definition}
\newtheorem*{notation}{Notation}
\newtheorem{ass}{Assumption}
\newtheorem{rem}{Remark}
\newtheorem{xmpl}{Example}
\numberwithin{ass}{section}
\numberwithin{prop}{section}
\numberwithin{rem}{section}
\numberwithin{lem}{section}
\usepackage[colorlinks,pdfdisplaydoctitle]{hyperref}

\title[Preconditioning the prior in Bayesian inverse
problems]{Preconditioning the prior to overcome saturation in Bayesian
  inverse problems} 
\author{Sergios Agapiou}
\email{Sergios.Agapiou@warwick.ac.uk}
\address{Mathematics Institute, University of warwick, Coventry CV4
  7AL, United Kingdom}
\author{Peter Math\'e}
\email{peter.mathe@wias-berlin.de}
\address{Weierstra{\ss} Institute for Applied Analysis and
  Stochastics, Mohrenstra\ss e 39, 10117 Berlin,  Germany}
\date{Version: \today}
\date{\today:  \currenttime}
\begin{document}
\begin{abstract}
  We study Bayesian inference in statistical linear inverse problems
  with Gaussian noise and priors 
  in Hilbert space. We focus our interest on the posterior
  contraction rate in the small noise limit. Existing results suffer from a certain
  saturation phenomenon, when the data generating element is too
  smooth compared to the smoothness inherent in the prior. We show how
  to overcome this saturation in an empirical Bayesian framework by
  using a non-centered data-dependent prior. The center is obtained
  from a preconditioning regularization step, which provides us with
  additional information to be used in the Bayesian
  framework. We use general techniques known from regularization
    theory. {To highlight the significance of the findings we
    provide several examples. In particular, our approach
      allows to obtain and, using preconditioning improve after
      saturation, {minimax} rates of
      contraction established in previous studies. We also establish {minimax}
      contraction rates in cases which have not been considered so
      far.} 
\end{abstract}
\maketitle

\section{Setup}
\label{sec:setup}

We consider the following linear equation in {real} Hilbert space
\begin{equation*}
  \yd = K x + \delta\eta,
\end{equation*}
where $K\colon X\to Y$ is a linear operator acting between 
the real separable Hilbert spaces~$X$ and $Y$, $\eta\sim \mathcal
N(0,\Sigma)$ is an additive centered Gaussian noise, {and $\delta>0$ is a scaling constant modelling the size of the noise}. Here, the covariance operator $\Sigma:Y\to Y$ is a
self-adjoint and positive definite bounded linear operator. 
We formally pre-whiten this equation and get
\begin{equation*}
  \label{eq:base-white-orig}
 \z =   \Sigma^{-1/2} \yd = \Sigma^{-1/2}K x + \delta\xi,
\end{equation*}
where now $\xi \sim \mathcal N(0,I)$ is Gaussian white noise. We
assign $\bart:= \Sigma^{-1/2}K$, and we assume that this is
bounded by imposing the condition  $\range(K)\subset \domain(\Sigma^{-1/2})$. We hence arrive to the data model
\begin{equation}
  \label{eq:base-white}
 \z =   T x + \delta\xi,
\end{equation}
and we consider the Bayesian approach to the statistical inverse
problem of finding $x$ from the observation $\z$. We assume Gaussian
priors on $x$, distributed according to~$\G(0,\frac{\delta^{2}}\alpha
  \prc)$, where $\prc: X\to
X$ is a positive definite, self-adjoint and trace class linear
operator, and $\alpha>0$ is a scaling constant. 
Linearity suggests that the posterior is also Gaussian and in this
paper we are interested in the asymptotic performance of the posterior
in the small noise limit, $\delta\to0$.

\subsection*{Squared posterior contraction}

Suppose that we observe data~$\z$ generated from the model
(\ref{eq:base-white}) for a fixed underlying true element $\xast\in X$
and  corresponding to a  noise level $\delta$. It is then reasonable
to expect that for small $\delta$ and for appropriate values of
$\alpha,$ the posterior Gaussian distribution will concentrate
around the true data-generating element~$\xast$. As we discuss below, this concentration will be driven by the \emph{squared
  posterior contraction} (SPC), given as
\begin{equation}
  \label{eq:sqpostconc}
  \spc:= \expect^{\xast}\expect^{\z}_{\alpha}\norm{\xast - x}{}^{2},
\end{equation}
where the outward expectation is taken with respect to the data
generating distribution,
that is, the distribution  generating~$\z$ when $\xast$ is given, and the inward expectation is taken with respect to the posterior
distribution, given data~$\z$ and having chosen a
parameter~$\alpha$. The Gaussian posterior distribution has a posterior mean, say~{$\xad=
\xad(\z;\alpha)$}, and a posterior covariance, say $\Cad$, which is
independent from the data~$\z$, and thus deterministic.
Then the inner expectation obeys the usual bias-variance decomposition
$$
\expect^{\z}_{\alpha}\norm{\xast - x}{}^{2} = \norm{\xast - \xad}{}^{2} + \tr{\Cad}.
$$
Applying the expectation with respect to the data generating
distribution, we obtain that
\begin{equation*}
\expect^{\xast}\expect^{\z}_{\alpha}\norm{\xast - x}{}^{2} =
\expect^{\xast}\norm{\xast - \xad}{}^{2} + \tr{\Cad}.
\end{equation*}
The quantity~$\expect^{\xast}\norm{\xast - \xad}{}^{2}$ represents the
mean integrated squared error (MISE) of the posterior mean viewed as an estimator of $\xast$, and
it has again a bias-variance decomposition into squared bias $b^2_{\xast}(\alpha):= \norm{\xast -
  \expect^{\xast}\xad}{}^{2}$ and
estimation variance~$\Vad:=\expect^{\xast}\norm{\xad-\expect^{\xast}\xad}{}^2$. We have thus decomposed the squared
  posterior contraction into respectively the \emph{squared bias}, the \emph{estimation
  variance}, and the \emph{spread in the posterior distribution}
\begin{equation}
  \label{eq:base-decomposition}
\spc(\alpha,\delta)= b^2_{\xast}(\alpha)+ \Vad + \tr{\Cad}. 
\end{equation}
We emphasize here, that the decomposition remains valid in the more general case of non-centered Gaussian priors.
  
It is clear, that if possible the hyper-parameter $\alpha$ should be chosen in a way that optimizes the SPC.
This raises several questions and challenges.

First, how do the estimation variance~$\Vad$ and the posterior spread
$\tr{\Cad}$ relate?
In previous studies, these quantities appear to be either of the same order, see proof of \cite[Thm 4.1]{MR2906881}, or the posterior spread dominates the estimation variance, see proofs of \cite[Thm 4.3]{MR3215928} and \cite[Thm 2.1]{MR3031282}. As
  was first highlighted in~\cite{BayesIP}, there is a natural
  relation~$\Vad \leq \tr{\Cad},$ whenever the prior is centered.

The \emph{posterior contraction rate} is concerned with the concentration rate of the posterior distribution around the truth, in the small noise
  limit $\delta\to 0$, and given a prior distribution. {It is well
    known, that the square root of the convergence rate of SPC is a
    posterior contraction rate (see for example \cite[Section
    7]{MR3084161})}. Given the prior scaling assumed here, SPC decays
  to zero provided that the parameter $\alpha$ 
  is chosen such that $\alpha=\alpha(\delta)\to0$ in an appropriate manner. 
 The study of this decay was the subject of the
papers~\cite{MR2906881,MR3084161,MR3031282,MR3215928}. 
 The obtained rates of convergence depend on the 
relationship between the regularity of the data-generating element
$\xast$ and the regularity inherent in the prior (see
\cite[\S~2.4]{DS13} for details on the regularity of draws from Gaussian measures in Hilbert space). The general message is that 
if the prior regularity matches the regularity of $\xast$, then the
convergence rate of SPC is the minimax-optimal rate even without
rescaling the prior, that is for the scaling considered here, $\alpha$ should be chosen
to be equal to $\delta^2$. If there is a mismatch between the prior
regularity and the regularity of the truth, then the minimax rate can
be achieved by appropriately rescaling the prior. If the prior is
smoother than the truth, then there exists an a priori parameter
choice rule $\alpha=\alpha(\delta)$ such that
$\frac{\delta^2}{\alpha}\to\infty$ as $\delta\to0$, which gives the optimal rate. If however the prior is rougher than the truth, then the minimax rate can be achieved by appropriate choices $\alpha=\alpha(\delta)$ such that $\frac{\delta^2}{\alpha}\to0$ as $\delta\to0$, in general only up to a maximal smoothness of {$\xast$}. As quoted in \cite{MR2906881}, rescaling can make the prior arbitrarily 'rougher' but not arbitrarily 'smoother'. 
A closer look at the situation reveals, and we shall highlight this
 in our subsequent analysis, that the estimation bias, which is part
 of the SPC in~(\ref{eq:base-decomposition}), is responsible for this phenomenon.
Bounds for the bias depend on the 
  inter-relation between the underlying solution smoothness and the
  capability of the chosen (Tikhonov-type since we have Gaussian priors) reconstruction by means of
  $\xad$  to take it into   account. The
  capability of such a scheme to take smoothness into account is
  called \emph{qualification} of the scheme, whereas the limited decay rate of the
  bias, as $\alpha\to 0$, due to the chosen reconstruction scheme, is
  called \emph{saturation} of the scheme. Details will be given
  below.

Finally, the optimal choices $\alpha=\alpha(\delta)$ depend
  on the regularity of $\xast$, which is in practice unknown. In
  the literature there have been two strategies to overcome these
  difficulties, both in the simplified setting of the white noise model (that is, the case
  $K=\Sigma=I$). The first one is to attempt to learn the correct
  scaling from the data, either by  using a maximum likelihood
  empirical Bayes approach, or by a fully hierarchical approach. This
  has been studied in~\cite{MR3044507}, where the results show that in
  both approaches the minimax rate is achieved but again up to a maximal
  regularity of the truth (which surprisingly is smaller than the one
  for the oracle type choice of $\alpha$). The second strategy is to not
  rescale the prior but rather attempt to learn the correct regularity
  from the data, again either using a maximum likelihood empirical
  Bayes or a fully hierarchical approach. This is the topic
  of \cite{knapiketal-adaptive},  where indeed the authors show that
  the minimax rate is achieved by both of the approaches. The last method seems to address both the issue of saturation and of choosing $\alpha$, however, all of the methods mentioned in this paragraph
 can be  difficult to implement. On the one hand as it is shown in \cite{ABPS14}, the implementation of the hierarchical approach in non-trivial problems is problematic in high dimensions and for small noise, while on the other hand the above empirical Bayes approaches involve solving an optimization problem which also becomes difficult for non-trivial problems.

\subsection*{Paradigm}

Here we
consider the following alternative paradigm. 
Suppose we want to use a Gaussian prior with covariance $\prc$, and
prior mean~$\prm$ to gain posterior inference for the
problem~(\ref{eq:base-white}).
The question we address is whether the prior center~$\prm$ has a significant impact on
the posterior contraction rate, and if so, how to choose it
'optimally' in the presence of data.
The subsequent analysis will show that the convergence rate of SPC will
improve by an appropriate  adjustment of the prior if the underlying
solution~$\xast$ has large smoothness.  In terms of the previous
discussion, for a prior of fixed smoothness this enables us to make \emph{a
priori} choices of $\alpha=\alpha(\delta)$ such that the posterior
contraction rate is minimax-optimal even for higher smoothness of $\xast,$ by
choosing an appropriate center~$\prm$ of the prior distribution.
 The proposed {re-centering~$\prm=\prm(\z;\alpha)$} of the
prior depends on the data~$\z$ and the parameter~$\alpha$, it is not
static. However,  it can easily be managed by a regularization step
preprocessing the Bayes step. We anticipate these results
in the following Figure~\ref{fig1}. This figure highlights the results as described in \S~\ref{sec:moderate-examples}.
\begin{figure}[htp]
  \center{
    \includegraphics[type=pdf, ext=.pdf, read=.pdf,
    width=0.7\columnwidth]{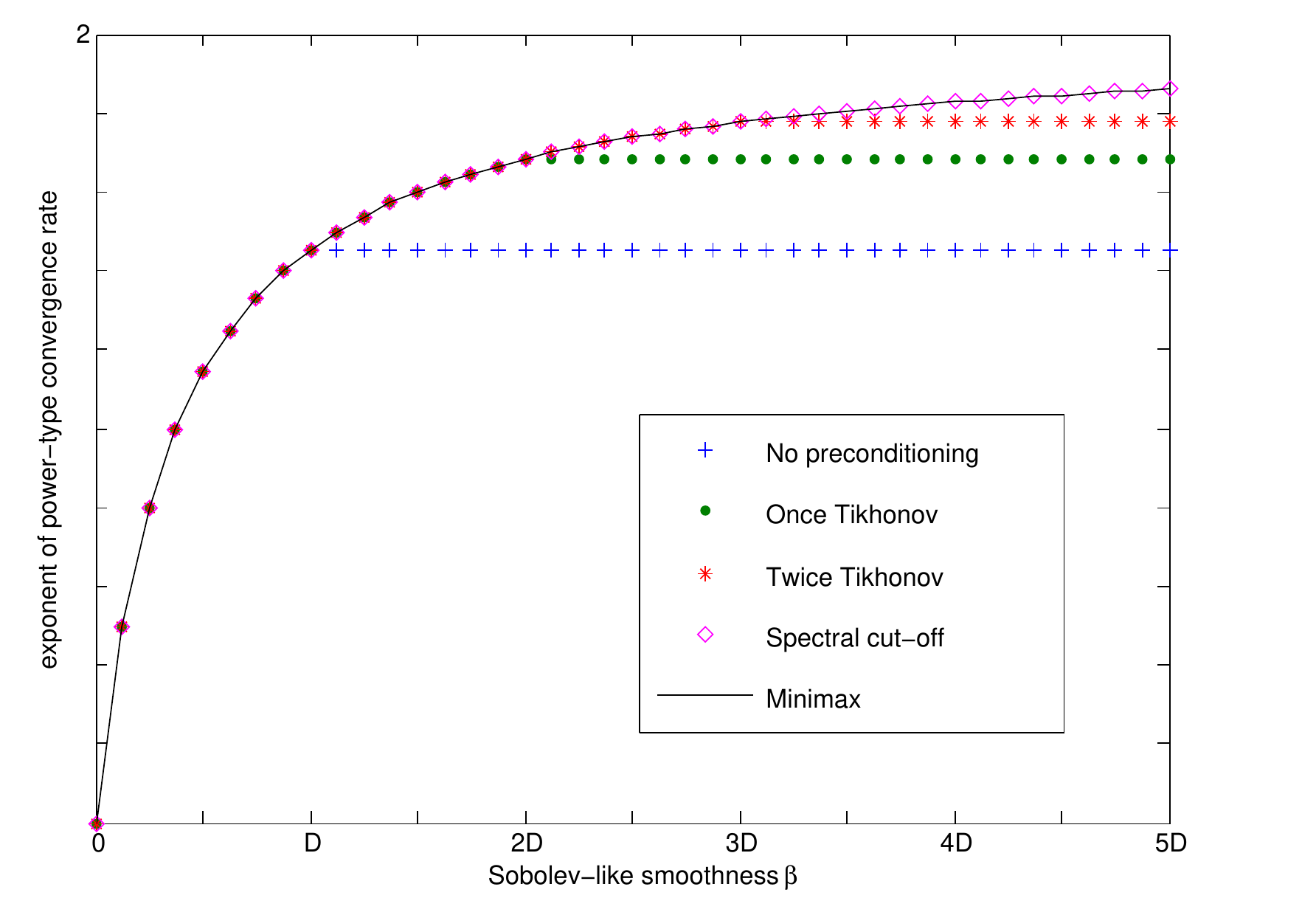}
  }
  \caption{Exponents of convergence rates of SPC plotted against Sobolev-like smoothness of the truth $\psmooth$, for different methods of choosing the prior mean $\mad$, in the moderately ill-posed problem discussed in \S~\ref{sec:moderate-examples}. We set $D:=1+2\prr+2p$, the saturation point when no preconditioning of the prior mean is used. Rates calculated for $\prr=0.5,~ p=1$.}  \label{fig1}
\end{figure}

We capture the advantages in a few lines:
  \begin{itemize}
  \item the user may choose a (centered) Gaussian prior of arbitrary
    smoothness;
\item after observing data~$\z$, a prior center, say $\prm=\prm(\z;\alpha)$ is determined by some
  deterministic regularization;
\item if this preprocessing regularization has enough
  qualification, then the posterior distribution will contract order
  optimally regardless of the solution smoothness. If not, then 
the contraction rate is at least as good as the rate corresponding to a centered prior.
  \item this preprocessing step has no effect on the parameter
  choice; so \emph{any} choice~{$\alpha =
  \alpha(\delta;\z)$} which yields
{'optimal'} contraction without
    preprocessing will retain this property, and will eventually
    extend this optimality property for higher solution smoothness. 
      \end{itemize}

\subsection*{Outline}

In order to explain the new paradigm we first study the impact of
using a non-centered prior to the posterior mean and covariance. Then
we specify the prior centering by means of using a linear
regularization in Eq.~(\ref{eq:prm}), as such is known from regularization theory. Next, we provide explicit representations of the quantities involved
in the subsequent analysis, the posterior mean, the posterior
covariance, and formulas for the bias and estimation variance, see
Eq.~(\ref{ali:pom-representation})--(\ref{ali:vad}). 

The main results are given in Section~\ref{sec:link}, after confining
ourselves to the case of commuting operators~$\prc$ and $\tast$,
expressed in terms of a specific link condition. We first derive
bounds for the estimation bias in Proposition~\ref{pro:bias-main}, and
these bounds are crucial for overcoming the saturation. Then we
introduce the net posterior spread in \S~\ref{sec:cad-control},
which is the unscaled version of the posterior spread, and we highlight its properties. We then combine to obtain our main result on the convergence of SPC, which is Theorem \ref{thm:main}.

To emphasize the significance of our results we discuss in
Section~\ref{sec:examples} specific
examples {some of which} were previously studied in~\cite{MR2906881, MR3215928, MR3031282}. 
In order to facilitate the reading of the study we postpone all proofs 
to the final Section~\ref{sec:proofs}.

\section{Setting the pace}
\label{sec:pace}

As mentioned above, we shall discuss a preprocessing of the
prior by choosing it non-central, that is, we will introduce a shift $\prm$,
such that the prior will be Gaussian with $\cN(\prm,{\frac{\delta^2}\alpha}\prc)$. In particular, we are interested in understanding the impact of the shift $\prm$ on the convergence rate of SPC.
For the reader's convenience, we start with deriving formulas for the posterior mean~$\xad$ in this
context. 

We first recall the representation of the posterior mean~$\pom$ and posterior
covariance~$\poc$ when a centered prior $\cN(0,\frac{\delta^{2}}\alpha\prc)$
is used. In this case we know, see for example~\cite{AM84, MR1009041}, {that almost surely with respect to the joint distribution of $(x,\z)$ the posterior is Gaussian, $\cN(\pom,\poc)$, for
\begin{align}\label{eq:cpom}
  \pom &= \prc^{1/2}\lr{\alpha I+\bastb}^{-1}B^{\ast}\z,\\
\intertext{and}
\poc &=  {\delta^{2}} \prc^{1/2} \lr{\alpha I+\bastb}^{-1}\prc^{1/2},
\end{align}
where we define {the compact operator} $B:=T \prc^{\frac12}$.
Re-centering the prior towards $\prm$ does not affect the
posterior covariance~$\poc$. To obtain the shift in the posterior mean we
rewrite~(\ref{eq:base-white}) as
$$
\z - T \prm = T(x - \prm) + \delta \xi
$$
Thus if $x\sim \cN(\prm,\prc)$ then $x-\prm\sim \cN(0,\prc)$. We are
in the usual context with centered prior but new data $\z - T
\prm$. This gives the representation for the posterior mean (shifting
back towards $\prm$) as
\begin{align*}
\xad &= \prm + \prc^{1/2}\lr{\alpha I+\bastb}^{-1}B^{\ast}(\z - T
\prm)\\
&= \prc^{1/2}\lr{\alpha I+\bastb}^{-1}B^{\ast}\z + \prm -
\prc^{1/2}\lr{\alpha I+\bastb}^{-1}B^{\ast}T\prm \\
&= \prc^{1/2}\lr{\alpha I+\bastb}^{-1}B^{\ast}\z + \prc^{1/2}\lr{I
  -\lr{\alpha I+\bastb}^{-1}\bastb }\prc^{-1/2}\prm\\
&=  \prc^{1/2}\lr{\alpha I+\bastb}^{-1}B^{\ast}\z + \prc^{1/2}\sa(\bastb)\prc^{-1/2}\prm,
\end{align*}
where we introduce the function~$\sa(t) = \alpha/(\alpha +
t),\ \alpha,t>0$, applied to the self-adjoint operator~$\bastb$ by
using spectral calculus.

It is well-understood from previous Bayesian analysis that a static choice
of $\prm$ will not have impact on the posterior contraction. However,
within our new paradigm we choose any regularization scheme~$\ga$ and
assign the prior center as
\begin{equation}
  \label{eq:prm}
 { \prm(\z;\alpha)} := \mad =  \prc^{1/2}\ga(\bastb)B^{\ast}\z.
\end{equation}

We introduce linear regularization schemes as follows.
\begin{Def}
  [linear regularization]\label{def:linreg}
{Let $b=\norm{\bastb}{}$.}
A family of piece-wise continuous functions~$\ga\colon (0,b]\to
  \real,\ \alpha>0$, is called \emph{regularization filter} with residual
  function~$\ra(t) = 1 - t \ga(t),\ \alpha, 0<t \leq b $, if
  \begin{enumerate}
  \item\label{it:r-bound} $\sup_{{0 < t \leq b}} \abs{\ra(t)}\leq \gamma_{0}$, for all $\alpha>0$,
\item \label{it:r-lim}$\lim_{\alpha\to 0} \ra(t)=0$ {for each $0< t
  \leq b$}, and
\item\label{it:g-bound} $\sup_{{0 < t \leq b}}\abs{\ga(t)}\leq \gamma_{\ast}/\alpha$, for all $\alpha>0$.
  \end{enumerate}
\end{Def}

The above  requirements are the ones which are typically imposed on a linear
regularization scheme, see for example~\cite{MR2318806}.
\begin{rem}
The element~$\prm(\z;\alpha)$ belongs to the Cameron-Martin space of the prior, that is, the subspace $\domain(\prc^{-\frac12})$ of ~$X$,
almost surely with respect to the joint distribution of $(x,\z)$. To see this combine the first assertion of Definition \ref{def:linreg} with the fact that the operator~$\prc^{1/2}$ is Hilbert--Schmidt. As a side remark, we mention that this means that the Gaussian prior measures corresponding to any parameter $\alpha$, or even any regularization filter $g_\alpha$, are absolutely continuous with respect to each other. 
\end{rem}
\begin{rem}
  The last assertion in Definition \ref{def:linreg} is actually stronger than the one required
  in~\cite{MR2318806}, but it is a convenient strengthening, and most
  known regularization schemes obey this stronger bound.
\end{rem}
\begin{rem}\label{rem:comparison}
 We use the following convention: if no preconditioning is
 used, that is,\ if $\ga(t) \equiv 0$, then we assign the constant
 function~$\ra(t) \equiv 1$, in order to simplify the comparison of
 the different settings. Specifically, without preprocessing we would
 naturally (and statically)  use $\prm:=0$ as the prior mean.
\end{rem}
\begin{xmpl}
  [Tikhonov regularization]\label{xmpl:Tikhonov}
One of the commonly used regularization schemes is Tikhonov
regularization, in which case the filter~$\ga$ is given as~$\ga(t) =
  1/(\alpha + t),\ \alpha,t>0$. Notice that in the case $\prm=0$, the posterior mean as given in Eq. (\ref{eq:cpom}), has the form of the right hand side in Eq. (\ref{eq:prm}) with $\ga$ being the Tikhonov filter.
\end{xmpl}\begin{rem}\label{rem:sa}
We fix once and
  for all, as above the function~$s_{\alpha}(t) = \alpha/(\alpha+t),$ that is,
  the residual function for 
  Tikhonov regularization. {This is done in order to distinguish the (Tikhonov) regularization in the posterior mean due to the use of a Gaussian prior, from the chosen
  regularization for the prior preconditioning.}
\end{rem}
\begin{xmpl}
  [$k$-fold Tikhonov regularization]\label{xmpl:k-Tikhonov}
We may iterate Tikhonov regularization, starting from the trivial
element~$x_{0,\alpha}=0$ as
$$
x_{j,\alpha}^{\delta} := x_{j-1,\alpha}^{\delta} + \lr{\alpha I +
  \bastb}^{-1}B^{\ast}\lr{\z - B x_{j-1,\alpha}^{\delta}},\quad j=1,\dots,k.
$$
For $k=1$ this gives Tikhonov regularization. 
The resulting linear
regularization is given by the function~$g_{k,\alpha}:= \frac 1 t\lr{1
- \lr{\frac \alpha {\alpha + t}}^{k}},\ t>0$, with corresponding residual
function~$r_{k,\alpha}= \lr{\frac \alpha {\alpha +t}}^{k},\ t>0$.
This regularization results in the prior center~$\mad = \prc^{1/2} x_{k,\alpha}^{\delta}=\prc^\frac12g_{k,\alpha}(\bastb)B^\ast \z$.

\end{xmpl}
\begin{xmpl}
  [spectral cut-off, truncated SVD]\label{xmpl:cutoff}
This is a versatile scheme, which requires to know the singular value
decomposition of the underlying operator. If this is available, then
we let~$\ga(t) = 1/t$, for $t\geq
  \alpha$ and $\ga(t)=0$ else.
\end{xmpl}

\medskip

We summarize the previous considerations and fix the notation which
will be used subsequently. Given prior mean $\mad$ 
from~(\ref{eq:prm}), we have that the posterior distribution is
Gaussian with posterior mean, denoted as $\xad$, given as
\begin{align}\label{ali:pom-representation}
 \xad &=  \prc^{1/2}\lr{\alpha I+\bastb}^{-1}B^{\ast}\z +
 \prc^{1/2}\sa(\bastb)\prc^{-1/2}\mad,
\intertext{and posterior covariance~$\poc:=\Cad$ with} 
\Cad &=  {\delta^{2}} \prc^{1/2} \lr{\alpha I+\bastb}^{-1}\prc^{1/2}.
  \label{ali:cad-representation}
\end{align}
Since we aim at controlling the squared posterior contraction, we have
that the spread is given as $\tr{\Cad}$ and we next give expressions for the corresponding
estimation bias and estimation variance.
\begin{lem}
  \label{lem:bias-estvar}
Let~$\xad$ be as in~(\ref{ali:pom-representation}). Then the
estimation bias and estimation variances, with posterior mean as
estimator, are 
\begin{align}
  b_{\xast}(\alpha) &= \norm{\prc^{1/2} \sa(\bastb)\ra(\bastb)
    \prc^{-1/2}\xast}{},\quad \alpha >0, 
  \label{ali:bias}
\intertext{and}
\Vad &=  \delta^{2} \tr{\lr{I + \alpha \ga(\bastb)}^{2} \lr{\alpha I +
    \bastb}^{-2} \bastb\prc},\ \alpha>0, \label{ali:vad}
\end{align}respectively.
\end{lem}
\begin{prop}\label{pro:vadlessspread}
  Let the prior center be obtained from any regularization (with
  corresponding constant~$\gamma_{\ast}$). Then we have that
\begin{equation}
  \label{eq:vadlesscad}
  \Vad \leq (1 + \gamma_{\ast})^{2} \tr{\Cad}.
\end{equation}
Consequently we have that
\begin{equation*}
\expect^{\xast}\norm{\xast - \xad}{}^{2} \leq  \spc(\alpha,\delta) \leq b^2_{\xast}(\alpha) + \lr{1 + (1 + \gamma_{\ast})^{2} } \tr{\Cad}.
\end{equation*}
\end{prop}
\begin{rem}
  The above analysis extends the previous bound from~\cite[Eq.~(12)]{BayesIP}
to the present context (note that without preprocessing we have that
$\gamma_{\ast}=0$).
{
We also note that the decay of the squared posterior contraction
cannot be faster than the minimax error for statistical estimation.}
\end{rem}
We thus have that in order to (asymptotically) bound the squared posterior contraction, we only need
to establish bounds for the bias and the posterior spread.
\section{Assumptions and main results}
\label{sec:link}
We are now ready to present our main results.
Before we do so, in \S~\ref{sec:concepts} we introduce several concepts used in our formulation. 
First, we introduce \emph{link
  conditions}, relating the two operators appearing in the setting at hand. Then we introduce \emph{source sets}, which we use for expressing the regularity of the truth. Finally, we introduce the \emph{qualification} of  a regularization which quantifies its capability to take high smoothness into account. We then present our bounds for the bias, the posterior spread and finally the squared posterior contraction in \S~\ref{sec:bounding-bias}, \S~\ref{sec:cad-control} and \S~\ref{sec:bounding-spq}, respectively.
\subsection{Link conditions, source sets and qualification}\label{sec:concepts}
We call a function~$\varphi\colon (0,\infty)\to\real^{+}$ an  \emph{index
  function} if it is {a continuous non-decreasing function which can be extended to take the value zero at the origin.}
\begin{rem}
{The property of interest of an index function is its asymptotic behaviour near the origin.
  In some cases the 'native' index function
  is not defined on $(0,\infty)$, but only on some sub-interval, say
  $(0,\bar t)$. Consider for example the logarithmic function~$\varphi(t) =
  \log^{-\mu}(1/t),\ 0< t < \bar t=1$ with $\phi(0)=0$. Then one can extend the function $\phi$ at some
  interior point $0< t_{0} < \bar t$ in an
  increasing way, for instance as $\varphi(t) = \varphi(t_{0}) +
  (t-t_{0}),\ t\geq t_{0}$. By doing so we ensure that the extended
  function shares the same asymptotic properties near zero, that is,\ as $t\searrow
  0$. In all subsequent (asymptotic)
  considerations it suffices to have such extensions, and this will not be mentioned explicitly.}
  \end{rem}
  
To simplify the outline of the study we confine ourselves to commuting
operators~$\prc$ and $\tast$. Specifically we do this as follows.
\begin{ass}[link condition]
  \label{ass:link} There is an index function~$\psi$ such that
  \begin{equation}
    \label{eq:linl}
    \psi^{2}(\prc) = \tast.
  \end{equation}
\end{ass}

Along with the function~$\psi$ we introduce the function
\begin{equation}
  \label{eq:theta-def}
  \Theta_{\psi}(t) := \sqrt t \psi(t),\quad t>0. 
\end{equation}
We draw the following consequence.
\begin{lem}\label{lem:f-props}
  Let $\psi$ be the index function for which Assumption~\ref{ass:link}
  holds.  Then the operators~$\prc$ and $\tast$ commute. Moreover we
  have that
$$
\Theta_{\psi}^{2}(\prc) = \bastb.
$$
\end{lem}
Following the last lemma, we set
\begin{equation}
  \label{eq:f-interpolation}
  f(s) := \lr{\lr{\Theta_{\psi}^{2}}^{-1}(s)}^{1/2},\quad s>0.
\end{equation}
We stress that the function~$f$ is an index function, since the
function~$\Theta_{\psi}$ was one. Moreover, the function~ $\Theta_{\psi}^{2}$
is strictly increasing, such that its inverse is a well defined
strictly increasing index function. Finally, as can be drawn from
Lemma~\ref{lem:f-props}, we have that under Assumption~\ref{ass:link}
it holds
\begin{equation}
  \label{cor:\prc12}
  \prc^{1/2} = f(\bastb). 
\end{equation}
\begin{rem}\label{rem:knapik-psi}
We remark the following about Assumption~\ref{ass:link}.
\begin{itemize}
\item 
The case that the operator~$T$ is the identity is not covered by this
assumption. This would require the function~$\psi\equiv1$,
which does not constitute an index function. However, for the
subsequent analysis we shall only use Lemma~\ref{lem:f-props}. 
As seen from~(\ref{cor:\prc12}) we obtain that~$\Theta_{\psi}(t) =
 \sqrt{t},\ t>0$, in this case.

\item If the prior~$\prc$ has eigenvalues with multiplicities higher
    than one, then by Assumption \ref{ass:link} the operator~$\tast$ also needs to have eigenvalues with higher multiplicities, since taking functions of operators preserves or increases the multiplicities of the eigenvalues. This is not realistic, hence one should choose a prior covariance
with eigenvalues of multiplicity one. This can be achieved by a
slight perturbation of the original choice.
\end{itemize}
\end{rem}

In order to have a handy notation we agree to introduce the following
partial ordering between index functions.
\begin{notation}\label{de:prec}
  Let $f,g$ be index functions. We say that $f \prec g$ if the
  quotient $g/f$ is {non-decreasing}. In other words $f \prec g$ if $g$ decays to zero faster than~$f$.
\end{notation}

For bounding the bias below we shall assume that the smoothness of the underlying true data-generating element $\xast$, is given as a
source set with respect to $\prc$.
\begin{Def}[source set]\label{de:sset}
  There is an index function~$\varphi$ such that
$$
\xast\in\af := \set{x,\quad x= \varphi(\prc)w,\ \norm{w}{}\leq 1}.
$$ 
\end{Def}
By Lemma~\ref{lem:f-props} the source set~$\af$ can be
rewritten as 
$$
\af = \set{x,\quad x= \varphi(f^{2}(\bastb))w,\
  \norm{w}{}\leq 1},
$$
 with the function~$f$ from~(\ref{eq:f-interpolation}).
Furthermore, under Assumption~\ref{ass:link} the operators~$\prc$ and $\bastb$
commute, and hence the
bias representation from~(\ref{ali:bias}) simplifies to
\begin{equation}
  \label{eq:bias-commute}
  b_{\xast}(\alpha) = \norm{\ra(\bastb)\sa(\bastb)\xast}{}.
\end{equation}
Overall, if~$\xast\in\af$ then 
$$ 
b_{\xast}(\alpha)\leq
\norm{\ra(\bastb)\sa(\bastb)\varphi(f^{2}(\bastb))}{}= \sup_{0< t \leq
  \norm{\bastb}{}} \abs{\ra(t)}\sa(t)\varphi(f^{2}(t)).
$$
 We shall bound this in terms of the parameter~$\alpha>0$, which directs us to the notion of a \emph{qualification} of a
regularization, see~\cite{MR2318806}, again.
\begin{Def}
  [qualification] A regularization $\ga$ has qualification~$\varphi$
  with constant~$\gamma$,   for an index function~$\varphi$,  
  if
$$
\abs{\ra(t)}\varphi(t) \leq \gamma \varphi(\alpha),\quad \alpha>0,\quad 0< t
\leq \norm{\bastb}{}.$$
\end{Def}
The following result is a well-known consequence, see~\cite[Prop.~2.7]{MR2318806}
again, albeit important for the subsequent analysis. We shall use the
partial ordering from Definition~\ref{de:prec}. 
\begin{lem}
  \label{lem:phipsi}
  Let $\ga$ be a regularization with index function~$\varphi$ as a
  qualification (with constant $\gamma$). If $\psi$ is an index
  function for which $\psi \prec \varphi$ then $\psi$ is also a
  qualification (with constant $\gamma$).
\end{lem}
\begin{rem}
As seen from the above analysis of the bias, we shall apply this to the compound function~$\ra(t) \sa(t)$, which is
related to the compound regularization, obtained by pre-conditioning
and Tikhonov regularization. {Clearly, it is desirable to bound the bias by a function of $\alpha$ which decays to zero as quickly as possible.}
It is thus apparent, that a qualification $\varphi$ of a regularization
quantifies its capability to take smoothness, given in
terms of source sets, into account. \end{rem}

\begin{xmpl}
  [Tikhonov regularization]\label{xmpl:Tikhonov-quali} Tikhonov regularization has (maximal)
  qualification~$\varphi(t) =t,\ t>0$. Thus, if for an index
  function~$\psi$ we have that~$\psi(t)\prec t$ then $\psi$ is a
  qualification.  In particular, all concave index functions are
  qualifications of Tikhonov regularization with constant~$\gamma=1$.
\end{xmpl}
\begin{xmpl}
  [spectral cut-off]\label{xmpl:cutoff-quali} Spectral cut-off has arbitrary
  qualification, since $\ra(t) = 0,\ t \geq \alpha$ and $\ra(t)=1$ elsewhere. Hence
$$
\ra(t) \varphi(t) = 0\leq \varphi(\alpha),\ t\geq \alpha,\quad\text{
  and}\quad \ra(t) \varphi(t) \leq \varphi(\alpha),\ t\leq \alpha.
$$
\end{xmpl}

\begin{rem}
 { We immediately see from~(\ref{ali:bias}) that the qualification
  of the regularization in the bias, can be raised from $t$ (Tikhonov regularization) to $t^{k+1}$, if the
  residual function~$\ra$ of the regularization used for preconditioning the prior mean has qualification~$t^{k}$, as is the
  case for $k$-fold Tikhonov regularization, see Example~\ref{xmpl:k-Tikhonov}.  If preconditioning is
  done by spectral cut-off, then the regularization in the bias has arbitrary qualification.}
\end{rem}

\subsection{Bounding the bias}
\label{sec:bounding-bias}
We are now ready to present our bounds for the bias.
\begin{prop}\label{pro:bias-main}
  Suppose that $\xast\in \af $, and that $\mad$ uses a
  regularization~$\ga$ with constant~$\gamma_{0}$ bounding the corresponding residual function.
  \begin{enumerate}
  \item\label{it:low-smoothness} If $\varphi \prec \Theta_{\psi}^{2},$
    then $b_{\xast}(\alpha) \leq \gamma_{0}\varphi\lr{f^{2}(\alpha)},\
    \alpha >0.$ 
\item\label{it:high-smoothness-noprecon} If $ \Theta_{\psi}^{2} \prec
    \varphi$ and if there was no preconditioning, then there are
    constants~$c_{1}, c_{2}>0$ (depending on $\xast,
    \varphi, f^{2}$, and on $\norm{\bastb}{}$) such that 
    $c_1\alpha\leq b_{\xast}(\alpha) \leq c_{2}\alpha,\ 0<\alpha\leq 1.$
  \item\label{it:high-smoothness-precon} If $ \Theta_{\psi}^{2} \prec
    \varphi$ and if
    $t\mapsto \varphi\lr{f^{2}(t)}/t$ is a qualification for the regularization $\ga$ with constant $\gamma$, then
    $b_{\xast}(\alpha) \leq \gamma\varphi\lr{f^{2}(\alpha)},\ \alpha >0.$
  \end{enumerate}
\end{prop}
\begin{rem} \label{rem:knapik-bias} 
 We mention that the above two cases $\varphi \prec \Theta_{\psi}^{2}$ or $\Theta_{\psi}^{2} \prec
    \varphi$ are nearly disjoint, with $\varphi= \Theta_{\psi}^{2}$
    being the only common member. {Therefore the
  function~$\Theta_{\psi}^{2}$ may be viewed as the \emph{'benchmark
  smoothness'}.} {However, note that the items~(\ref{it:low-smoothness}) and~(\ref{it:high-smoothness-precon}) do not exhaust all possibilities since the function $\varphi\lr{f^{2}(t)}/t$ may not be a qualification for $\ga$ (in fact it may not even be an index function).} 
\end{rem}
\begin{rem}
  \label{rem:maxbound-bias}
  We stress that the bounds in
  item~(\ref{it:high-smoothness-noprecon}) show the saturation
  phenomenon in the bias if no preconditioning of the prior mean is
  used: for any sufficiently high smoothness the bias decays with the
  fixed rate $\alpha$. In other words, if no preconditioning of the
  prior is used, the best achievable rate of decay for the bias is
  linear. Item (\ref{it:high-smoothness-precon}) shows that
  appropriate preconditioning improves things, since for high
  smoothness the bias decays at the superlinear rate
  $\varphi(f^2(\alpha))$.
\end{rem}
\subsection{The net posterior spread}
\label{sec:cad-control}

Here we study the posterior spread, that is, the trace of the
posterior covariance 
from~(\ref{ali:cad-representation}), which will be needed for
determining the contraction rate.  In order to highlight the
nature of the spread in the posterior within the assumed Bayesian framework,
 we make the following definition, for a given equation
$\z = T x + \delta\xi$, with white noise~$\xi$, as considered 
in~(\ref{eq:base-white}). 
 \begin{Def}
   [net posterior spread]\label{def:spread}
The function
  $$
  \Sp(\alpha) := \tr{\prc^{1/2} \lr{\alpha I+\bastb}^{-1}\prc^{1/2}},\quad \alpha>0,
  $$
is called the \emph{net posterior spread}.
 \end{Def}
 Notice that with this function we have that $\tr{\Cad} =
\delta^{2}\Sp(\alpha)$.
Moreover, using the cyclic commutativity of the trace, we get that 
\begin{equation}\label{eq:netpostspread}\Sp(\alpha) = \tr{\lr{\alpha I+\bastb}^{-1}\prc}.
\end{equation}
With this more convenient representation at hand, we establish some fundamental
properties of the net posterior spread, {which are crucial for optimizing the converhence rate of SPC in the following subsection.}
\begin{lem}
  \label{lem:S-properties}\ 
  \begin{enumerate}
  \item The function~$\alpha \mapsto \Sp(\alpha)$ is strictly
    decreasing and continuous for $\alpha>0$.
  \item $\lim_{\alpha\to\infty}\Sp(\alpha)=0$, and
  \item $\lim_{\alpha\to 0}\Sp(\alpha)=\infty$.
  \end{enumerate}
\end{lem}

\subsection{Bounding the squared posterior contraction}
\label{sec:bounding-spq}

It has already been highlighted that the squared posterior contraction as
given in~(\ref{eq:sqpostconc}) is decomposed into the sum of the squared bias, estimation variance and posterior spread,
see~(\ref{eq:base-decomposition}). By
Proposition~\ref{pro:vadlessspread} we find that
$$
b^{2}_{\xast}(\alpha) + \delta^{2}\Sp(\alpha)\leq \spc(\alpha)
\leq b^{2}_{\xast}(\alpha) +  \lr{(1 +
  \gamma_{\ast})^{2} +1}  \delta^{2}\Sp(\alpha).
$$
In the asymptotic regime of  $\delta\to 0$, the size of SPC
is thus determined by the sum~$b^{2}_{\xast}(\alpha) +
\delta^{2}\Sp(\alpha)$. In \S~\ref{sec:bounding-bias} we have
established bounds for the bias. Here we just constrain to the case
where, given that $\xast\in\af$, the preconditioning is such that the
size of the bias is bounded by (a multiple of)
$\varphi(f^{2}(\alpha))$, see Proposition~\ref{pro:bias-main}. {Since
  $b^{2}_{\xast}(\alpha)$ is bounded by a non-decreasing function of
  $\alpha$ which decays to zero as $\alpha\searrow 0$, while by Lemma
  \ref{lem:S-properties} the function~$\Sp(\alpha)$ is strictly decreasing, continuous and onto the positive half-line},  
the SPC is 'minimized' by the choice of $\alpha$ which balances the
bound for the squared bias and the spread. {This choice clearly exists and is unique and hence we immediately arrive to our main result.}
\begin{thm}
  \label{thm:main}
Let $\varphi$ be any index function, and assume that item~(\ref{it:low-smoothness}) or
item~(\ref{it:high-smoothness-precon}) in Proposition~\ref{pro:bias-main}
hold. 
Consider the equation
\begin{equation}
  \label{eq:balance}
  \varphi^{2}(f^{2}(\alpha))= \delta^{2}\Sp(\alpha).
\end{equation}
The equation~(\ref{eq:balance}) is uniquely solvable, and let
$\aast=\aast(\varphi,\delta)$ be the solution.
For $\xast\in\af$ we have that~$\spc(\aast,\delta) = \mathcal{O}(
\varphi^{2}(f^{2}(\aast)))$ as $\delta\to 0$.
\end{thm}
{The importance of this theorem will become apparent in the next section. In many specific
cases, the
obtained contraction rates of the SPC correspond to known minimax
rates in statistical inverse problems. This can be seen in Propositions ~\ref{prop:exp-exp}, \ref{prop:severe}, \ref{prop:modmild}
and~\ref{prop:expexp} below. For general link conditions and general source conditions, minimax
rates and in particular lower bounds are scarce. Here we mention the
study~\cite{MR2240642}, where the linking function~$\psi$ is of power
type, and the smoothness function~$\varphi$ is
assumed to be concave.
\begin{rem}
  \label{rem:maxrate-SPC}
As emphasized in Remark~\ref{rem:maxbound-bias}, if no preconditioning is used, the best rate at which
the bias can decay is linear. {This effect, which is called saturation
(of Tikhonov regularization), was discussed in a more general context in regularization theory, and
we mention the study~\cite{MR2112789}. }

So, if no preconditioning is present, then the left
hand side in~(\ref{eq:balance}) at best decays as $\alpha^{2}$.
We conclude that the best rate of decay of the SPC which can be established
without preconditioning is $\aast^{2}$, where $\aast$ is obtained from
balancing $\alpha^{2} = \delta^{2} \Sp(\alpha)$. {Balancing actually
gives (up to some constant) the minimum value, as it was shown in
Lemma~2.4 in the same reference.}
\end{rem}

\section{Examples and discussion}\label{sec:examples}
We now study {several examples, some},  which are standard in the
literature, and some which exhibit {new features.} {Our aim is to}
demonstrate the simplicity of our method for deriving rates of
posterior contraction and most importantly  the benefits of
preconditioning the prior.  

{Before we proceed we stress the following fact, which is not so accurately
spelled out in other studies. It is important to distinguish the \emph{degree of ill-posedness of the operator}~$T$ which governs
equation~(\ref{eq:base-white}), and which expresses the decay of
its singular numbers, from the \emph{degree of ill-posedness of the problem},
which corresponds to the operator~$T$ \emph{and} the solution smoothness,
and thus regards the achievable contraction rate. As we will see
in~\S~\ref{sec:severe-analytic} below, the problem can have a signifficantly different degree of ill-posedness
than the operator $T$.}

{We first consider two examples which concern Sobolev-like smoothness of the
truth. We recover the moderately and severely ill-posed problems,  as for example studied in~\cite{MR2906881}, and ~\cite{MR3031282,
  MR3215928}, respectively. {Then, we consider another two examples which concern analytic-type smoothness of the truth, which to our knowledge have not been studied before. First, we once more study the moderately ill-posed operator problem, which we will see that under analytic-type smoothness of the truth leads to what we call a \emph{mildly ill-posed} problem. Then, we study a problem with severely ill-posed operator, which as we will see, under analytic-type smoothness of the truth leads to a \emph{moderately ill-posed} problem}.

In all of the examples, the operators $C_0$ and $T^\ast T$ are
simultaneously diagonalizable in an orthonormal basis $\{e_j\}$ which is complete in $X$,
$C_0$ has spectrum that decays as $\{j^{-1-2\prr}\}, \prr>0$, while
$T^\ast T$ can either have spectrum that decays polynomially
(moderately ill-posed operator case) or exponentially (severely ill-posed operator case).
\begin{notation}\label{def:bigo}
  Given two positive functions $k, h:\mathbb R^+\to\mathbb R^+$, we use
   $k\asymp h$ to denote that $k= \bigo(h)$ and $h=\bigo(k)$ as $s\to
  0$. Furthermore, the notation $h(s)\gg k(s)$, means that
  $k(s)=\bigo(h(s) s^{\mu})$ as $s\to0$ for some positive power
  $\mu>0$.
\end{notation}
\subsection{Smoothness relative to the prior}
\label{sec:smooth-examples}

In the first two examples, we present posterior contraction rates under
the assumption that we have the a priori knowledge that the truth
belongs to the Sobolev ellipsoid 
\begin{equation}
  \label{eq:sbeta}
  S^\psmooth=\{x\in X:
\sum_{j=1}^\infty j^{2\psmooth}x_j^2\leq1\},
\end{equation}
for some $\psmooth>0$ and where  $x_j:=\langle x,
e_j\rangle$. Relative to~$\prc$, the index function defining the source set
$\af$ in Definition \ref{de:sset}, is in this case
$\varphi(t)=t^{\frac{\psmooth}{1+2\prr}}$.

In the third example, we present posterior contraction rates under
analytic smoothness of the truth, that is, we assume that we have the
a priori knowledge that the truth belongs to the ellipsoid 
\begin{equation}
  \label{eq:abeta}
\mathcal{A}^\psmooth=\{x\in X: \sum_{j=1}^\infty e^{2\psmooth j}x_j^2\leq1\}, 
\end{equation}
for some $\psmooth>0$. In this case, the index function defining the
source set $\af$ in Definition \ref{de:sset}, is
$\varphi(t)=\exp(-\psmooth t^{-\frac{1}{1+2\prr}})$.

\subsection{Moderately ill-posed {operator} under Sobolev smoothness}
\label{sec:moderate-examples}
We consider the moderately ill-posed setup studied in~\cite{MR2906881}, in which the operator $T^\ast T$ has spectrum which
decays as $\{j^{-2p}\}$ for some $p\geq0$, and thus the singular numbers of~$\bastb$
decay as $s_{j}(\bastb)\asymp j^{-(1 + 2\prr +
  2p)}$.

In the present case Assumption~ \ref{ass:link}, which expresses the operator $T^\ast T$ as a function
of the prior covariance operator $C_0$, is satisfied for
$\psi^2(t)=t^{\frac{2p}{1+2\prr}}$. Next, we find that the function
$\Theta_{\psi}$ in (\ref{eq:theta-def}), which expresses the operator
$B^\ast B$ as a function of $C_0$, is given as
$\Theta_{\psi}(t)=t^\frac{1+2\prr+2p}{2(1+2\prr)}$, {hence the benchmark smoothness is $\Theta_{\psi}^{2}(t) = t^{\frac{1 +
  2\prr + 2p}{1+2\prr}}$}. Finally, we
have that the function $f$ in (\ref{eq:f-interpolation}), which
expresses $C_0$ as a function of $B^\ast B$ is given by~{$f(s)=s^{\frac{1+2\prr}{2(1+2\prr+2p)}}$}.

\subsubsection*{Bounding the bias}
We now have all the ingredients required to bound the bias. 
{The
following result is an immediate consequence of Proposition
\ref{pro:bias-main} and the considerations of the previous paragraph}.

\begin{prop}\label{prop:moderate-moderate}
  Suppose that $x^\ast\in S^\psmooth,$ for some $\psmooth>0$. Then as $\alpha\to0$:
  \begin{enumerate}
  \item\label{it:low-gamma} If $\psmooth\leq1+2\prr+2p$, and independently of whether
    preconditioning of the prior is used or not, we have that
    $b_{\xast}(\alpha)=\mathcal{O}(\alpha^{\frac{\psmooth}{1+2\prr+2p}})$;
  \item\label{it:high-gamma-noprecon}if $\psmooth>1+2\prr+2p$ and no preconditioning of the
    prior is used, then $b_{\xast}(\alpha)\asymp \alpha$;
  \item\label{it:high-gamma-precon} if $\psmooth>1+2\prr+2p$ and $\mad$ uses a regularization
    $g_\alpha$ with qualification
    $t^{\frac{\psmooth-1-2\prr-2p}{1+2\prr+2p}}$, then
    $b_{\xast}(\alpha)=
   \mathcal{O}(\alpha^{\frac{\psmooth}{1+2\prr+2p}})$.
  \end{enumerate}
\end{prop}
We stress here that our contribution is
item~(\ref{it:high-gamma-precon}). In particular,
item~(\ref{it:high-gamma-precon}) 
implies that if we choose the prior mean $\mad$ using the $k$-fold
Tikhonov regularization filter ({cf. Example~\ref{xmpl:k-Tikhonov}}), which has maximal qualification $t^k$, then for
$\psmooth\leq (k+1)(1+2\prr+2p)$ we have that $b_{\xast}(\alpha)=
\mathcal{O}(\alpha^{\frac{\psmooth}{1+2\prr+2p}})$, that is, the saturation in the bias
is delayed. If we choose $\mad$ using the spectral cut-off
regularization filter, which as we saw in Example \ref{xmpl:cutoff-quali} has arbitrary qualification, then for any $\psmooth>0,$ we have that
$b_{\xast}(\alpha)=\mathcal{O}(\alpha^{\frac{\psmooth}{1+2\prr+2p}})$, that is,
there is no saturation in the bias.

\subsubsection*{Bounding the SPC}

To see the impact of this result to the SPC rate, we apply Theorem
\ref{thm:main}. In order to do so, we first need to calculate the net posterior spread which in this case is {such that
$\Sp(\alpha)\asymp\alpha^{-\frac{1+2p}{1+2\prr+2p}}$}, see \cite[Thm 4.1]{MR2906881}. Concatenating we get
the following result.

\begin{prop}\label{prop:exp-exp}
  Suppose that $x^\ast\in S^\psmooth, $ $\psmooth>0$. Then as $\delta\to0$:
  \begin{enumerate}
  \item\label{it:spc-low-gamma} if $\psmooth\leq1+2\prr+2p$ and independently of whether
    preconditioning of the prior is used or not, for
    $\alpha=\delta^{\frac{2(1+2\prr+2p)}{1+2p+2\psmooth}}$ we have
    that $SPC=\mathcal O(\delta^{\frac{4\psmooth}{1+2\psmooth+2p}})$;
  \item\label{it:spc-high-gamma-noprecon} if $\psmooth>1+2\prr+2p$ and no preconditioning of the
    prior is used, then for any choice $\alpha=\alpha(\delta, \psmooth)$
    we have that $SPC\gg \delta^\frac{4\psmooth}{1+2\psmooth+2p}$;
  \item\label{it:spc-high-gamma-precon} if $\psmooth>1+2\prr+2p$ and $\mad$ uses a regularization
    $g_\alpha$ with qualification
    $t^{\frac{\psmooth-1-2\prr-2p}{1+2\prr+2p}}$, for
    $\alpha=\delta^{\frac{2(1+2\prr+2p)}{1+2p+2\psmooth}}$ we have
    that $SPC=\mathcal O(\delta^{\frac{4\psmooth}{1+2\psmooth+2p}})$.
      \end{enumerate}
\end{prop}
As before, our contribution is item~(\ref{it:spc-high-gamma-precon}),
which in particular implies that 
if we choose the prior mean $\mad$ using the $k$-fold
Tikhonov regularization filter, then for $\psmooth\leq
(k+1)(1+2\prr+2p)$ we achieve the optimal (minimax) rate
$\delta^{\frac{4\psmooth}{1+2\psmooth+2p}}$, that is the saturation in the SPC
is also delayed. If we choose $\mad$ using the spectral cut-off
regularization filter, then for any $\psmooth\geq0$ we achieve the
optimal rate $\delta^{\frac{4\psmooth}{1+2\psmooth+2p}}$, that is, there is
no saturation in the SPC! Note that the optimal scaling of the prior, as a function of the noise level $\delta$,
is the same whether we use preconditioning or not.
We depict the findings in Figure~\ref{fig1}.

\subsection{Severely ill-posed {operator} under Sobolev smoothness}
\label{sec:severe-examples}
We now consider the severely ill-posed setup studied in
\cite{MR3215928, MR3031282}, in which the operator $T^\ast T$ has spectrum which
decays as $\{e^{-2qj^b}\}$ for some $q,b>0$, and thus the singular numbers of~$\bastb$
decay as $s_{j}(\bastb)\asymp j^{-(1+
  2\prr)}e^{-2qj^b}$.

In this case Assumption~ \ref{ass:link}, which expresses the operator $T^\ast T$ as a function
of the prior covariance operator $C_0$, is satisfied for
$\psi^2(t)=\exp(-2qt^{-\frac{b}{1+2\prr}})$. Next, we find that the function
$\Theta_{\psi}$ in (\ref{eq:theta-def}), which expresses the operator
$B^\ast B$ as a function of $C_0$, is given as
$\Theta_{\psi}(t)=t^\frac12\exp(-qt^{-\frac{b}{1+2\prr}})$, and hence the benchmark smoothness is $\Theta^2_{\psi}(t)=t\exp(-2qt^{-\frac{b}{1+2\prr}})$. Finally, we
have that as $t\to0$, the function $f$ in (\ref{eq:f-interpolation}) which
expresses $C_0$ as a function of $B^\ast B$ behaves as
$f(s)\sim (\log(s^{-\frac{1}{2q}}))^{-\frac{1+2\prr}{2b}}$, see Lemma \ref{lem:sevinv} in Section \ref{sec:proofs}.

\subsubsection*{Bounding the bias}

In this example we have that $\Theta_{\Psi}^2(t)$ decays exponentially, while $\varphi(t)$ polynomially, hence for any Sobolev-like smoothness of the truth $\psmooth$, it holds $\varphi \prec \Theta_{\psi}^{2}$. In other words, even without preconditioning there is no saturation in the bias and we are always in case (\ref{it:low-smoothness}) in Proposition
\ref{pro:bias-main}. However, our theory still works and we can easily derive the rate for the bias and SPC. The next result follows immediately from the considerations in the previous paragraph and Proposition \ref{pro:bias-main}.
\begin{prop}
  Suppose that $x^\ast\in S^\psmooth,$ $\psmooth>0$.
  Then independently of whether
    preconditioning of the prior is used or not, we have that
    $b_{\xast}(\alpha)=\mathcal{O}\big((\log(\alpha^{-1}))^{-\frac{\psmooth}b}\big)$, as $\alpha\to 0$.
\end{prop}

\subsubsection*{Bounding the SPC}
We now apply Theorem
\ref{thm:main} in order to calculate the SPC rate. Again, we first need to calculate the net posterior spread, which in this case is  such that {
$\Sp(\alpha)\asymp\frac1\alpha(\log(\alpha^{-1}))^{-\frac{2a}b}$}, see \cite[Thm 4.2]{MR3215928}. We prove the following result, which agrees with \cite[Thm 2.1]{MR3031282} and \cite[Thm 4.3]{MR3215928}.

\begin{prop}\label{prop:severe}
  Suppose that $x^\ast\in S^\psmooth,$ $\psmooth>0$.
 Then independently of whether
    preconditioning of the prior is used or not, for any $\sigma>0$, any parameter choice rule $\alpha=\alpha(\delta)$ such that 
    $\delta^2(\log(\delta^{-2}))^{\frac{2\psmooth-2\prr}b}\leq\alpha\leq\delta^{2\sigma},$ gives the rate $SPC=\mathcal O((\log(\delta^{-2}))^{-\frac{2\psmooth}{b}}))$, as
    $\delta\to 0$.
\end{prop}

\subsection{Moderately ill-posed operator under analytic smoothness}
\label{sec:moderate-analytic}
We now consider the moderately ill-posed operator setup studied in \S~\ref{sec:moderate-examples}
with the difference that {here} we assume that we have the a priori knowledge that the truth has a certain analytic smoothness. 
The functions $\psi, \Theta_{\psi}$ and $f$ which have to do with the
relationship between the forward operator and the prior covariance are
as in \S~\ref{sec:moderate-examples}, but the
function $\varphi$ which {describes analytic smoothness of the truth
  as in (\ref{eq:abeta}}), is now $\varphi(t)=\exp(-\psmooth
t^{-\frac{1}{1+2a}})$. In particular, since $\varphi$ is exponential
while the benchmark smoothness $\Theta^2_{\psi}$ is {of power type, we are always in the high smoothness case $\Theta_{\Psi}^2\prec\varphi$.}

\subsubsection*{Bounding the bias} The following is an immediate consequence of
Proposition \ref{pro:bias-main} and the considerations in the previous paragraph.
\begin{prop}
Suppose that $x^\ast\in\mathcal{A}^\psmooth$, for some $\psmooth>0$. Then as $\alpha\to 0$:
\begin{enumerate}
\item\label{it:expexp-saturation} if no preconditioning is used,  $b_{\xast}(\alpha)\asymp \alpha$;
\item\label{it:expexp-high} if $\mad$ uses a regularization
    $g_\alpha$ with qualification
    $\exp(-\psmooth t^{-1})$, then we have that $b_{\xast}(\alpha)=\mathcal{O}(\exp(-\psmooth\alpha^{-\frac{1}{1+2\prr+2p}}))$.
\end{enumerate}
\end{prop}
\begin{rem}
  \label{rem:power-exp-no-precon}
If no preconditioning is used, the bias convergence rate is always saturated. 
The qualification as formulated in item~(\ref{it:expexp-saturation})
is a \emph{sufficient condition}, while the actual form can be
calculated easily. The given form highlights that exponential type
qualification is required to overcome the limitation of the power type
prior covariance in order to treat analytic smoothness.  We stress here that such qualification is hard to achieve. For example, iterated Tikhonov can never achieve such exponential qualification, while even Landweber iteration which has qualification $t^\nu$, for any $\nu>0$, only achieves this qualification for values $\psmooth$ which are not too big. On the other hand, $\exp(-\psmooth t^{-1})$ is a qualification for spectral cut-off for any positive value of $\psmooth$.
\end{rem}

\subsubsection*{Bounding the SPC}
We again apply Theorem \ref{thm:main} in order to calculate the SPC rate. The net posterior spread is as in \S~\ref{sec:moderate-examples}, 
{$\Sp(\alpha)\asymp\alpha^{-\frac{1+2p}{1+2\prr+2p}}$}. We prove the
following result, using the convention from Definition~\ref{def:bigo}.

\begin{prop}\label{prop:modmild}
Suppose that $x^\ast\in \mathcal{A}^\psmooth, $ $\psmooth>0$. Then as $\delta\to0$:
  \begin{enumerate}
  \item\label{it:spc-high-gamma-noprecon-mild} if no preconditioning of the
    prior is used, then for any choice $\alpha=\alpha(\delta, \psmooth)$
    we have that $SPC\gg \delta^2$;
  \item\label{it:spc-high-gamma-precon-mild} if $\mad$ uses a regularization
    $g_\alpha$ with qualification
    $\exp(-\psmooth t^{-1})$, for
    $\alpha=(\log(\delta^{-1/\psmooth}))^{-(1+2\prr+2p)}$ we have
    that $SPC=\mathcal O(\delta^2(\log(\delta^{-1}))^{1+2p})$.
  \end{enumerate}
\end{prop}

\begin{rem}
  \label{rem:prop4.6}
We stress that according to item (\ref{it:spc-high-gamma-noprecon-mild}), without preconditioning we have that $\delta^2/SPC$ decays at an algebraic rate, while the optimal achievable (also minimax)
  rate is of power two up to some logarithmic factor.
  Since the optimal achievable rate in this case is of power two up to logarithmic factors, it is
  reasonable to call such problems \emph{mildly ill-posed}, as
  they are almost well-posed.
  \end{rem}

\subsection{Severely ill-posed operator under analytic smoothness}
\label{sec:severe-analytic}
We now consider the severely ill-posed operator setup studied in \S~\ref{sec:severe-examples}
with the difference that {here} we assume that we have the a priori knowledge that the truth has a certain analytic smoothness. For simplicity, we concentrate on the case $b=1$, which corresponds for example to the Cauchy problem for the Helmholtz equation, see \cite[Section~5]{MR3215928} for details.

The functions $\psi, \Theta_{\psi}$ and $f$ which have to do with the
relationship between the forward operator and the prior covariance are
as in \S~\ref{sec:severe-examples} for the value $b=1$, but the
function $\varphi$ which {describes analytic smoothness of the truth as in (\ref{eq:abeta}}), is now $\varphi(t)=\exp(-\psmooth t^{-\frac{1}{1+2a}})$. In particular, since both $\varphi$ and the benchmark smoothness $\Theta^2_{\psi}$ are exponential, unlike \S~\ref{sec:severe-examples} we now have a saturation phenomenon. 

\emph{Bounding the bias.} The following is an immediate consequence of
Proposition \ref{pro:bias-main} and the considerations in the previous paragraph.
\begin{prop}
Suppose that $x^\ast\in\mathcal{A}^\psmooth$, for some $\psmooth>0$. Then as $\alpha\to 0$: 
\begin{enumerate}
\item\label{it:expexp-low} if $\psmooth\leq2q$ and independently of whether preconditioning of the prior is used or not, we have that $b_{\xast}(\alpha)=\mathcal{O}(\alpha^{\frac{\psmooth}{2q}})$;
\item\label{it:expexp-saturation} if $\psmooth>2q$ and no preconditioning is used $b_{\xast}(\alpha)\asymp\alpha$;
\item\label{it:expexp-high} if $\psmooth>2q$ and $\mad$ uses a regularization
    $g_\alpha$ with qualification
    $t^{\frac{\psmooth-2q}{2q}}$, then we have that $b_{\xast}(\alpha)=\mathcal{O}(\alpha^{\frac{\psmooth}{2q}})$.
\end{enumerate}
\end{prop}
The benefits of preconditioning are once more clear and can be seen in item (\ref{it:expexp-high}). If for example we choose the prior mean $\mad$ using the $k$-fold Tikhonov regularization filter, then for $\psmooth\leq(k+1)2q$ we have that $b_{\xast}(\alpha)=\mathcal{O}(\alpha^{\frac{\psmooth}{2q}})$, that is the saturation in the bias is delayed. If we use spectral cut-off, then there is no saturation at all.

\emph{Bounding the SPC.} We again apply Theorem \ref{thm:main} in order to calculate the SPC rate. The net posterior spread is as in \S~\ref{sec:severe-examples}, {
$\Sp(\alpha)\asymp\frac{1}{\alpha}(\log(\alpha^{-1}))^{-2\prr}$.} We prove the following result.

\begin{prop}\label{prop:expexp}
Suppose that $x^\ast\in \mathcal{A}^\psmooth, $ $\psmooth>0$. Then as $\delta\to0$:
  \begin{enumerate}
  \item\label{it:spc-low-gamma-expmod} If $\psmooth\leq2q$ and independently of whether
    preconditioning of the prior is used or not, for
    $\alpha=\delta^{\frac{2q}{\psmooth+q}}$ we have
    that $SPC=\mathcal O(\delta^{\frac{2\psmooth}{\psmooth+q}})$;
  \item\label{it:spc-high-gamma-noprecon-expmod} if $\psmooth>2q$ and no preconditioning of the
    prior is used, then for any choice $\alpha=\alpha(\delta, \psmooth)$
    we have that $SPC\gg \delta^\frac{2\psmooth}{\psmooth+q}$;
  \item\label{it:spc-high-gamma-precon-expmod} if $\psmooth>2q$ and $\mad$ uses a regularization
    $g_\alpha$ with qualification
    $t^{\frac{\psmooth-2q}{2q}}$, for
    $\alpha=\delta^{\frac{2q}{\psmooth+q}}$ we have
    that $SPC=\mathcal O(\delta^{\frac{2\psmooth}{\psmooth+q}})$.  \end{enumerate}
\end{prop}
The benefits of preconditioning  can again be seen in item (\ref{it:expexp-high}). If for example we choose the prior mean $\mad$ using the $k$-fold Tikhonov regularization filter, then for $\psmooth\leq(k+1)2q$ we achieve the optimal (minimax) rate $\delta^{\frac{2\psmooth}{\psmooth+q}}$, that is the saturation in the SPC is delayed. If we use spectral cut-off, then there is no saturation at all. Note again that the optimal scaling of the prior, as a function of the noise level $\delta$, is the same whether we use preconditioning or not.

As anticipated, the features of this example, and in particular the polynomial rates of convergence, are characteristic of moderately ill-posed problems.

\subsection{{Summary and discussion}}
\label{sec:summary-outline}

We succinctly summarize the above examples, in which we confined to
power-type decay of the spectrum of the prior~$\prc$, that is,\ $s_{j}(\prc)\asymp
j^{-(1+ 2\prr)},\ j=1,2,\dots$, for some $\prr>0$.

First in \S~\ref{sec:moderate-examples} and \S~\ref{sec:severe-examples}, we specified the solution element to belong to some Sobolev-type
ball as in (\ref{eq:sbeta}), characterized
by~$\psmooth>0$. 
The distinction between moderately and severely ill-posed problems
then comes from the decay of the singular numbers of the operator~$T$
governing equation~(\ref{eq:base-white}). We outline the previous
results in Table~\ref{tab:summary}.
\begin{table}[ht]
\renewcommand{\arraystretch}{2}\addtolength{\tabcolsep}{-0pt}
  \centering
  \begin{tabular}{|| r | c || c | c || }
    \hline\hline
&& $s_{j}(\tast) \asymp j^{- 2p}$ & $s_{j}(\tast) \asymp e^{- 2 q
  j^{b}}$\\\hline\hline
link & $\psi$ & $t^{p/(1+2\prr)}$ & $\exp\lr{-2 q t^{-b/(1 + 2\prr)}}$\\
benchmark & $\Theta_{\psi}^{2}$ & $t^{(1 + 2\prr + 2p)/(1 + 2 \prr)}$
& $t\exp\lr{-2 q t^{-b/(1 + 2\prr)}}$\\
saturation & $\varphi = \Theta_{\psi}^{2}$ &  $\psmooth = 1 + 2 \prr +
2 p$ & always $ \varphi \prec \Theta_{\psi}^{2}$\\
contraction & SPC & $\delta^{4\psmooth/(1 + 2 \prr + 2 p)}$ & $\log^{- 2\psmooth/b}(1/\delta)$
\\\hline\hline
  \end{tabular}
\vspace*{.5\baselineskip}
  \caption{Outline of SPC rates for Sobolev-type smoothness of the truth, $\varphi(t) =t^{\psmooth/(1 + 2 \prr)},\ t>0$.}
  \label{tab:summary}
\end{table}

Then in \S~\ref{sec:moderate-analytic} and \S~\ref{sec:severe-analytic}, we considered analytic type smoothness of the truth as in (\ref{eq:abeta}), again characterized by $\psmooth>0$. As commented earlier on, to our knowledge we are the first to study these examples. Our findings show that the overall
problem degree of ill-posedness can be significantly different than the degree of ill-posedness of the operator. We outline the results in Table \ref{tab:summary2}.

\begin{table}[ht]
\renewcommand{\arraystretch}{2}\addtolength{\tabcolsep}{-0pt}
  \centering
  \begin{tabular}{|| r | c || c | c || }
    \hline\hline
&& $s_{j}(\tast) \asymp j^{- 2p}$  & $s_{j}(\tast) \asymp e^{- 2 q
  j^{b}}$\\\hline\hline
link & $\psi$ & $t^{p/(1+2\prr)}$&  $\exp\lr{-2 q t^{-b/(1 + 2\prr)}}$\\
benchmark & $\Theta_{\psi}^{2}$ 
&$t^{(1 + 2\prr + 2p)/(1 + 2 \prr)}$& $t\exp\lr{-2 q t^{-b/(1 + 2\prr)}}$\\
saturation & $\varphi = \Theta_{\psi}^{2}$ & always $\Theta_{\psi}^{2}
\prec \varphi$& $\psmooth=2q$\\
contraction & SPC & $\delta^2 \log^{1+ 2p}(1/\delta)$&  $\delta^{2\psmooth/(\psmooth+q)}$
\\\hline\hline
  \end{tabular}
\vspace*{.5\baselineskip}
  \caption{Outline of SPC rates for analytic-type smoothness of the truth, $\varphi(t)=\exp(-\psmooth t^{-\frac{1}{1+2a}}),\ t>0$.}
  \label{tab:summary2}
\end{table}

The rates exhibited in
Tables \ref{tab:summary} and \ref{tab:summary2}, correspond
to the minimax rates as given in~\cite[Tbl.~1]{MR2421941}.

\section{Proofs and auxiliary results}
\label{sec:proofs}
\begin{proof}
  [Proof of Lemma~\ref{lem:bias-estvar}]
We first express the element~$\xad$ in terms of $\z$.
\begin{align*}
  \xad &=  \prc^{1/2}\lr{\alpha I+\bastb}^{-1}B^{\ast}\z +
 \prc^{1/2}\sa(\bastb)\prc^{-1/2}\mad\\
&=  \prc^{1/2}\lr{\alpha I+\bastb}^{-1}B^{\ast}\z +
 \prc^{1/2}\sa(\bastb)\ga(\bastb)B^{\ast}\z\\
&= \prc^{1/2}\left[\lr{\alpha I+\bastb}^{-1} + \sa(\bastb)\ga(\bastb) \right]B^{\ast}\z.
\end{align*}
We notice that 
$$
\lr{\alpha I+\bastb}^{-1} + \sa(\bastb)\ga(\bastb) =
\lr{\alpha I+\bastb}^{-1} \lr{I + \alpha \ga(\bastb)}.
$$
The expectation of the posterior mean with respect to the distribution generating $\z$ when $\xast$ is given, is thus

$$
\expect^{\xast} \xad = \prc^{1/2}\left[\lr{\alpha I+\bastb}^{-1} \lr{I + \alpha \ga(\bastb)}\right]\bastb \prc^{-1/2}\xast.
$$
For the next calculations we shall use that
\begin{align*}
  I - \lr{\alpha I+\bastb}^{-1} &\lr{I + \alpha \ga(\bastb)}\bastb\\
&= 
 \lr{\alpha I+\bastb}^{-1} \alpha\lr{I - \ga(\bastb)\bastb}\\
&= \sa(\bastb)\ra(\bastb).
\end{align*}
Therefore we rewrite
\begin{align*}
  \xast - \expect^{\xast} \xad & = 
\prc^{1/2}\left[ I - \lr{\alpha I+\bastb}^{-1} \lr{I + \alpha \ga(\bastb)} \bastb\right] \prc^{-1/2}\xast\\
&= \prc^{1/2}\sa(\bastb)\ra(\bastb)\prc^{-1/2}\xast,
\end{align*}
which proves the first assertion.
The variance is $\expect^{\xast} \norm{\xad -
  \expect^{\xast}\xad}{}^{2}$, and this can be written as
in~(\ref{ali:vad}), by using similar reasoning as for the bias term.
\end{proof}
\begin{proof}
  [Proof of Proposition~\ref{pro:vadlessspread}]
We notice that $\norm{I + \alpha\ga(\bastb)}{}\leq 1 + \gamma_{\ast}$,
which gives
\begin{align*}
 \Vad &= \delta^{2} \tr{\lr{I +\alpha \ga(\bastb)}^{2} \lr{\alpha I +
     \bastb}^{-2} \bastb\prc} \\
& \leq \delta^{2} \lr{1 + \gamma_{\ast}}^{2}\tr{\lr{\alpha I +
     \bastb}^{-2} \bastb\prc} 
\end{align*}
Since $\norm{\lr{\alpha +   \bastb}^{-1} \bastb}{}\leq 1$ we see that 
$$
 \Vad \leq  \lr{1 + \gamma_{\ast}}^{2}\delta^{2} \tr{\lr{\alpha I +
     \bastb}^{-1}\prc} = \lr{1 + \gamma_{\ast}}^{2}\tr{\Cad},
$$
and the proof is complete.
\end{proof}

\begin{proof}[Proof of Lemma~\ref{lem:f-props}]
  Since $\prc$ has finite trace, it is compact, and we use the
  eigenbasis (arranged by decreasing eigenvalues) $u_{j},\
  j=1,2,\dots$ Under Assumption~\ref{ass:link} this is also the
  eigenbasis for $\tast$. If $t_{j},\ j=1,2,\dots$ denote the
  eigenvalues then we see that
$$
\tast = \sum_{j=1}^{\infty} \tau_{j} u_{j}\otimes u_{j}.
$$
Correspondingly, $\prc = \sum_{j=1}^{\infty}
\lr{\psi^{2}}^{-1}(\tau_{j}) u_{j}\otimes u_{j}$, which gives the
first assertion. Moreover, the latter representation yields that
$$
\prc^{1/2} = \sum_{j=1}^{\infty}
\lr{\lr{\psi^{2}}^{-1}(\tau_{j})}^{1/2} u_{j}\otimes u_{j},
$$
such that
\begin{align*}
  \bastb &= \prc^{1/2} \tast \prc^{1/2} \\
  & = \sum_{j=1}^{\infty}\lr{\lr{\psi^{2}}^{-1}(\tau_{j})}^{1/2}
  \tau_{j} \lr{\lr{\psi^{2}}^{-1}(\tau_{j})}^{1/2}u_{j}\otimes u_{j}\\
  & = \sum_{j=1}^{\infty}\lr{\lr{\psi^{2}}^{-1}(\tau_{j})}
  \tau_{j} u_{j}\otimes u_{j}\\
  &=
  \sum_{j=1}^{\infty}\psi^{2}\lr{\lr{\lr{\psi^{2}}^{-1}(\tau_{j})}}\lr{\lr{\psi^{2}}^{-1}(\tau_{j})}
  u_{j}\otimes u_{j}\\
  &=
  \sum_{j=1}^{\infty}\Theta_{\psi}^{2}\lr{\lr{\psi^{2}}^{-1}(\tau_{j})}
  u_{j}\otimes u_{j}\\
  &= \Theta_{\psi}^{2}\lr{\prc},
\end{align*}
and the proof is complete.
\end{proof}
\begin{proof}
  [Proof of Proposition~\ref{pro:bias-main}]
For the first item~(\ref{it:low-smoothness}), we
  notice that $\varphi \prec \Theta_{\psi}^{2}$ if and only
  if $\varphi(f^{2}(t))\prec t$.  The linear function $t \mapsto t$ is a
  qualification of Tikhonov regularization with constant
  $\gamma=1$. Thus, by Lemma~\ref{lem:phipsi} we have
$$
b_{\xast}(\alpha) \leq \norm{\ra(\bastb)}{}
\norm{\sa(\bastb)\varphi(f^{2}(\bastb))}{} \leq \gamma_{0}\varphi(f^{2}(\alpha)),
$$
which completes the proof for this case. For item~(\ref{it:high-smoothness-noprecon}), we have that 
$$b_{\xast}(\alpha)=\norm{\sa(\bastb)\xast}{}.$$
For any $0<\alpha\leq1$,  we have $\alpha+t\leq 1+t$, hence 
$$b_{\xast}(\alpha)=\alpha\norm{(\alpha
  I+\bastb)^{-1}\xast}{}\geq\alpha\norm{( I+\bastb)^{-1}\xast}{}.$$ We
conclude that there exists a constant
{$c_{1}=c_{1}(\xast,\norm{\bastb}{})$}, such that for small
$\alpha$ it holds  
$$
b_{\xast}(\alpha)\geq {c_{1}}\alpha.
$$ 
On the other hand, since $t\prec\varphi(f^{2}(t)),$ there exists a
constant {$c_{2}>0$} which depends only on the index functions $\varphi$, $f$ and on~{$\norm{\bastb}{}$}, such that
$$
b_{\xast}(\alpha)=\alpha\norm{(\alpha
  I+\bastb)^{-1}\xast}{}\leq\alpha\norm{(\bastb)^{-1}\varphi(f^2(\bastb))w}{}\leq
{c_{2}}\alpha.
$$
For item~(\ref{it:high-smoothness-precon}),  we have that
\begin{align*}
  b_{\xast}(\alpha) &\leq
  \norm{\ra(\bastb)\sa(\bastb)\varphi(f^{2}(\bastb))}{}\\
  & \leq
  \norm{\sa(\bastb)\bastb}{}\norm{\ra(\bastb)\varphi(f^{2}(\bastb))\lr{\bastb}^{-1}}{}\\
  &\leq \alpha \gamma \frac{\varphi(f^{2}(\alpha))}{\alpha} = \gamma
  \varphi(f^{2}(\alpha)),
\end{align*}
and the proof is complete.
\end{proof}
\begin{proof}[Proof of Lemma~\ref{lem:S-properties}]
The continuity is clear. For the monotonicity we use the representation (\ref{eq:netpostspread}) to get \begin{align*}
  \Sp(\alpha) - \Sp(\alpha^{\prime}) & = \tr{\lr{\alpha I+\bastb}^{-1}\prc} -
  \tr{\lr{\alpha^{\prime} + \bastb}^{-1}\prc}\\
&= \tr{\lr{\alpha I+\bastb}^{-1}(\alpha^{\prime} - \alpha) \lr{\alpha^{\prime} +
    \bastb}^{-1}\prc}\\
& = (\alpha^{\prime} - \alpha) \tr{\lr{\alpha I+\bastb}^{-1}\lr{\alpha^{\prime} +
    \bastb}^{-1}\prc}.
\end{align*}
The trace on the right hand side is positive. Indeed, if
$(s_{j}^{2},u_{j},u_{j})$ denotes the singular value decomposition of
$\bastb$ then this trace can be written as
$$
\tr{\lr{\alpha I+\bastb}^{-1}\lr{\alpha^{\prime} +
    \bastb}^{-1}\prc}
= \sum_{j=1}^{\infty} \frac{1}{\alpha + s_{j}^{2}} \frac{1}{\alpha^{\prime} +
  s_{j}^{2}} \scalar{\prc u_{j}}{u_{j}},
$$
where the right hand side is positive since the operator~$\prc$ {is positive definite}. Thus, if $\alpha
<\alpha^{\prime}$ then~$\Sp(\alpha) - \Sp(\alpha^{\prime})$ is positive, which proves the
first assertion.

The proof of the second assertion is simple, and hence omitted.
To prove the last assertion we use the partial
ordering of self-adjoint operators in Hilbert space, that is,\ we write
$A \leq B$ if $\scalar{A x}{x} \leq \scalar{B x}{x},\ x\in X$, for two
self-adjoint operators~$A$ and $B$. Plainly, with $a:=
\norm{\tast}{}$, we have that $\tast \leq a I$. Multiplying from
the left and right by $\prc^{1/2}$ this yields~$\bastb \leq a\prc$,
and thus for any $\alpha>0$ that $\alpha I + \bastb \leq \alpha I +
a\prc$. The function~$t\mapsto -1/t,\ t>0$ is operator monotone, which
gives
$\lr{\alpha I + a\prc}^{-1} \leq \lr{\alpha I +
  \bastb}^{-1}$. Multiplying from the left and right by $\prc^{1/2}$
again, we arrive at
$$
\prc^{1/2}\lr{\alpha I + a\prc}^{-1}\prc^{1/2} \leq \prc^{1/2}\lr{\alpha I +
  \bastb}^{-1}\prc^{1/2}.
$$
This in turn extends to the traces and gives that
$$
\tr{\prc^{1/2}\lr{\alpha I + a\prc}^{-1}\prc^{1/2}}\leq \tr{ \prc^{1/2}\lr{\alpha I +
  \bastb}^{-1}\prc^{1/2}} = \Sp(\alpha).
$$ Now, let us denote by $t_j,\
j\in\mathbb N$, the singular numbers of $\prc$, then we can bound 
$$
 \Sp(\alpha) \geq \tr{\lr{\alpha I + a\prc}^{-1}\prc}\geq
 \sum_{t_{j}\geq \alpha/a} \frac{t_{j}}{\alpha + a t_{j}}
\geq \frac 1 {2a} \#\set{j,\ t_{j}\geq \frac \alpha a}.
$$
{If $\Sp(\alpha)$ were uniformly bounded from above, then there would exist a finite natural number, say $N$, such that $t_{N} \geq \frac \alpha a
> t_{N+1}$, for $\alpha>0$ small enough. But this would  imply that $t_{N+1}=0$,
which contradicts the assumption that $\prc$ is positive definite.}
\end{proof}

\begin{lem}\label{lem:sevinv}
For $t>0$ let $\Theta^2_{\Psi}(t)=t\exp(-2qt^{-\frac{b}{1+2\prr}})$, for some $q, b, a>0$. Then for small $s$ we have $(\Theta^2_{\Psi})^{-1}(s)\sim(\log s^{-\frac1{2q}})^{-\frac{1+2\prr}{b}}$.
\end{lem}
\begin{proof}
Let \begin{equation}\label{eq:ser1} s=\Theta^2_{\Psi}(t)>0\end{equation} and observe that $t$ is small if and only if $s$ is small. Applying~\cite[Lemma 4.5]{MR3215928} for $x=t^{-1}$ we get the result.\end{proof}

\begin{proof}[Proof of Proposition \ref{prop:severe}]
In this example the explicit solution of Eq. (\ref{eq:balance}) in Theorem \ref{thm:main} is more difficult. However,  as discussed in \S~\ref{sec:bounding-spq}, it suffices to asymptotically balance 
 the squared bias and the posterior spread using an appropriate parameter choice $\alpha=\alpha(\delta)$. Indeed, under the stated choice of $\alpha$ the squared bias is of order 
\begin{align*}
(\log(\alpha^{-1}))^{-\frac{2\psmooth}b}\leq \sigma^{-\frac{2\psmooth}b}\log(\delta^{-2})^{-\frac{2\psmooth}b}
\end{align*}
while the posterior spread term is of order
\begin{align*}
\frac{\delta^2}{\alpha}(\log(\alpha^{-1}))^{-\frac{2a}b}\leq\log(\delta^{-2}))^{-\frac{2\psmooth}b}.
\end{align*}
\end{proof}

\begin{proof}[Proof of Proposition \ref{prop:modmild}] 
{According to the considerations in Remark \ref{rem:maxrate-SPC}, it
  is straightforward to check that without preconditioning the best
  SPC rate that can be established is
  $\delta^{\frac{4+8\prr+8p}{3+4\prr+6p}}$ which proves item
  (\ref{it:spc-high-gamma-noprecon-mild}).} In the preconditioned
case, {the explicit solution of Eq. (\ref{eq:balance}) in
  Theorem \ref{thm:main}, which in this case has the
  form\[\exp({-2\psmooth\alpha^{-\frac{1}{1+2\prr+2p}}})=\delta^2\alpha^{-\frac{1+2p}{1+2\prr+2p}},\]
  is again difficult}. However,  as discussed in \S~\ref{sec:bounding-spq}, it suffices to asymptotically balance 
 the squared bias and the posterior spread using an appropriate
 parameter choice $\alpha=\alpha(\delta)$. Indeed, using \cite[Lem
 4.5]{MR3215928} we have that the solution to the above equation
 behaves asymptotically as the stated choice of $\alpha$, and
 substitution gives the claimed rate.\end{proof}
 
\begin{proof}[Proof of Proposition \ref{prop:expexp}]
We begin with items  (\ref{it:spc-low-gamma-expmod}) and (\ref{it:spc-high-gamma-precon-expmod}). The explicit solution of Eq. (\ref{eq:balance}) in Theorem \ref{thm:main}, which in this case has the form\[\alpha^{\frac{\beta}q}=\frac{\delta^2}{\alpha}(\log(\alpha^{-1})^{-2\prr},\] is difficult. As discussed in \S~\ref{sec:bounding-spq}, it suffices to asymptotically balance 
 the squared bias and the posterior spread using an appropriate parameter choice $\alpha=\alpha(\delta)$. Indeed, under the stated choice of $\alpha$ both quantities are bounded from above by $\delta^{\frac{2\psmooth}{\psmooth+q}}$. {For item (\ref{it:spc-high-gamma-noprecon-expmod}),  according to the considerations in Remark \ref{rem:maxrate-SPC}, it is straightforward to check that without preconditioning the best SPC rate that can be established is $\delta^{\frac{4q}{\psmooth+q}}$.}
\end{proof}

\bibliographystyle{amsplain}
\bibliography{iterate}

\end{document}